\documentclass[11pt]{article}
\usepackage{extpfeil}
\usepackage[all]{xy}
\usepackage{amssymb} 
\usepackage[margin=1in]{geometry}
\usepackage{euscript}
\usepackage{mathtools}
\usepackage{amsthm}
\usepackage{mathrsfs}
\usepackage{url}
\usepackage{color}

\numberwithin{equation}{section}

\theoremstyle{plain}
	\newtheorem{thm}[equation]{Theorem}
	\newtheorem{prop}[equation]{Proposition}

	\newtheorem{cor}[equation]{Corollary}
	\newtheorem{lem/defn}[equation]{Lemma/Definition}

\theoremstyle{definition}
	\newtheorem{defn}[equation]{Definition}

	\newtheorem{const}[equation]{Construction}

\theoremstyle{remark}
	\newtheorem{rem}[equation]{Remark}

\def\nc{\newcommand}
\def\on{\operatorname}

\nc{\edit}[1]{\marginpar{\footnotesize{#1}}}

\nc{\C}{\mathbb{C}}
\nc{\Z}{\mathbb{Z}}
\nc{\PP}{\mathbb{P}}
\nc{\R}{\mathbb{R}}
\nc{\AAA}{\mathbb{A}}
\nc{\LL}{\mathbb{L}}
\nc{\OO}{\mathcal{O}}

\nc{\X}{\EuScript{X}}
\nc{\sZ}{\EuScript{Z}}

\nc{\id}{{\on{id}}}
\nc\Hom{{\on{Hom}}}
\nc\cone{{\on{cone}}}
\nc{\Rep}{{\on{Rep}}}
\nc\Ob{{\on{Ob}}}
\nc\Spec{{\on{Spec}}}
\nc\Mod{{\on{Mod}}}
\nc\coMod{{\on{coMod}}}
\nc\Perf{{\on{Perf}}}
\nc\End{{\on{End}}}
\nc{\into}{\hookrightarrow}
\nc{\tr}{\on{tr}}
\nc{\ev}{\on{ev}}
\nc{\im}{\on{im}}
\nc{\Mot}{\on{Mot}}
\nc{\pt}{\on{pt}}
\nc{\coker}{\on{coker}}
\nc{\rk}{\on{rank}}
\nc{\TOP}{\on{Top}_{\mathbb{C}}^{s}}
\nc{\gr}{\on{gr}}
\nc{\Catperf}{\text{Cat}^{\text{perf}}}
\nc{\Sym}{\on{Sym}}
\nc{\xra}{\xrightarrow}
\nc{\lra}{\xleftarrow}
\nc{\Bet}{\mathbf{Betti}_{X}}
\nc{\codim}{\on{codim}}
\nc{\Fred}{\on{Fred}}
\nc{\colim}{\on{colim}}
\nc{\KK}{{\bf K}}
\nc{\Sp}{\on{Sp}}
\nc{\onto}{\twoheadrightarrow}
\nc{\A}{\mathbb{A}}
\nc{\Aff}{\on{Aff}}
\nc{\SH}{\on{SH}}
\nc{\QCoh}{\on{QCoh}}
\nc{\Alg}{\on{Alg}}
\nc{\Br}{\on{Br}}
\nc{\ta}{\widetilde{\a}}
\nc{\Shv}{\on{Shv}}
\nc{\GG}{\mathbb{G}}
\nc{\red}{\color{red}}
\nc{\an}{\on{an}}
\nc{\D}{\on{D}}
\nc{\Pre}{\on{Pre}}
\nc{\qc}{\on{qc}}
\nc{\op}{\on{op}}
\nc{\shEnd}{{\mathcal End}}
\nc{\Sph}{\mathbb{S}}
\nc{\Top}{\on{Top}}
\nc{\Map}{\on{Map}}
\nc{\Vect}{\on{Vect}}
\nc{\holim}{\on{holim}}

\def\A{\mathcal{A}}

\def\a{\alpha}

\def\Perf{\on{Perf}}

\def\Sp{\on{Sp}}

\def\QCoh{\on{QCoh}}

\title{The geometry of filtrations}
\author{Tasos Moulinos}
\date{}

\begin{document}
\vspace{18mm} \setcounter{page}{1} \thispagestyle{empty}

\maketitle

\begin{abstract}
We display a symmetric monoidal equivalence between the stable $\infty$-category of filtered spectra, and quasi-coherent sheaves on $\AAA^1 / \GG_m$, the quotient  in the setting of spectral algebraic geometry of the flat affine line by the canonical action of the flat multiplicative group scheme. Via a Tannaka duality argument, we identify the underlying spectrum and associated graded functors with pull-backs of quasi-coherent sheaves along certain morphisms of stacks.  
\end{abstract}

\section{Introduction}
The idea of viewing filtered vector spaces as quasi-coherent sheaves on the stack $\AAA^1 / \GG_m$ goes back to Carlos Simpson in his work on non-abelian Hodge theory \cite{simpson1990nonabelian}. Since then, this paradigm has been used to great effect in exhibiting naturally occuring filtrations on e.g.\ the category of $\mathcal{D}_{X}$-modules over a scheme in characteristic zero (see for example \cite{ben2012loop}). 

In this note, we observe that this identification takes on a particularly neat form in the setting of spectral algebraic geometry over the sphere spectrum, where one takes the basic affine objects to be connective $E_{\infty}$-ring spectra in place of ordinary ``discrete" commutative rings. Gluing these together one obtains spectral schemes; more general homotopy colimit constructions give rise to stacks in this setting. To an arbitrary spectral stack $\EuScript{X}$, one can associate a theory of quasi-coherent sheaves on $\EuScript{X}$. Our main theorem is the following identification of filtered spectra with quasi-coherent sheaves on the spectral stack $\AAA^1 / \GG_m$:

\begin{thm}
\label{main}
There is an equivalence 
$$
\QCoh(\mathbb{A}^1 / \mathbb{G}_{m}) \simeq  \on{Rep}(\Z)
$$
of symmetric monoidal stable  $\infty$-categories. 
\end{thm}

Let us explain the notation on the right hand side of this equivalence. Given an $\infty$-category $\EuScript{C}$, one may define the $\infty$-category of \emph{filtered} objects of $\EuScript{C}$ as the functor category 
$$
\Rep(\Z; \EuScript{C}) := \on{Fun}((\Z, \leq)^{op}, \EuScript{C}).
$$ 
Here, $\Z$ is viewed as a partially ordered set and therefore a category with morphisms  $[n] \mapsto [m]$ when $n \leq m$. When $\EuScript{C}= \Sp$, the stable $\infty$-category of spectra, one obtains the notion of a filtered spectrum; we henceforth use the notation $\Rep(\Z)$ to denote the $\infty$-category of filtered spectra. This comes equipped with the "Day convolution" summetric monoidal structure: given filtered spectra $(A_n)_{n \in \Z}, (B_n)_{n \in \Z} $ their tensor product is the filtered spectrum which in degree $n$ is given by 
$$
\colim_{ n \leq i+j}  A_i \otimes_{\Sph} B_j
$$

\noindent We prove the following corollary: 

\begin{cor} \label{suhdude}
Let $\EuScript{X}$ be a spectral stack that admits a map $\EuScript{X} \to \AAA^1/ \GG_m$.  Then there is a natural filtration on the cohomology $\Gamma(X,\OO_\EuScript{X})$ of $X$, where $X = \Spec (\Sph) \times_{\AAA^1/ \GG_m} \EuScript{X}$ denotes the ``generic fiber" of $\EuScript{X}$. 
\end{cor}

The existence of a filtration on a spectrum is of computational value; as is well known (see e.g. \cite[Section 1.2.2]{lurie2016}) this gives rise to a spectral sequence computing its homotopy groups. Corollary \ref{suhdude} gives a method of producing natural filtrations on various spectra of algebro-geometric origin. For an example of this approach, see \cite{moulinos2019universal} where, for $X$ a derived scheme over the $p$-local integers, the authors exhibit the derived mapping stack $\Map(S^1, X)$ as the underlying object of a filtered stack; this induces the HKR filtration on $\on{HH}(X)$, the Hoschchild homology of $X$.

\vskip\baselineskip
\noindent{\bf Conventions and Notation.}
We work with $\infty$-categories throughout, following conventions in \cite{lurie2009, lurie2016}. We repeatedly use the notation $\Sph[X]$ to denote the suspension spectrum of a space $X$. We also occasionally make use of the notation $\Pre(\EuScript{C})$ and $\Pre_{\Sp}(\EuScript{C})$ to denote the $\infty$-category of presheaves and spectrum-valued presheaves on $\EuScript{C}$. We use $\on{CAlg}(\EuScript{C}):= \Alg_{E_\infty}(\EuScript{C}).$ to denote $E_{\infty}$-algebra objects in a symmetric monoidal $\infty$-category $\EuScript{C}^{\otimes}$.

\vskip\baselineskip
\noindent{\bf Acknowledgements.} I would like to warmly thank Marco Robalo and Markus Spitzweck for helpful conversations and ideas. I would also like to thank Bertrand To\"{e}n, from whom I first learned about the connection between $\AAA^1/\GG_m$ and filtrations. I began this project while in residence at the Mathematical Sciences Research Institute in Berkeley in Spring 2019, supported by the National Science Foundation under Grant No. DMS-1440140.  This work is supported by the grant NEDAG ERC-2016-ADG-741501.

\section{The geometric stack $\AAA^1 / \GG_m$}
\label{chen}
We recall the main players of the story in greater depth. In the setting of spectral algebraic geometry over $\Spec(\mathbb{S})$, and in fact over any $E_{\infty}$-ring $R$, there exist two notions of ``affine line". They are only equivalent rationally, i.e. over $\Spec (\mathbb{Q})$. Here we will focus our attention on the ``flat" affine line, which is defined as the spectral affine scheme $\Spec (\Sph[\mathbb{N}])$, where $\pi_{*}\Sph[\mathbb{N}] \cong (\pi_{*}\Sph)[t]$. This has the property that it pulls back, via the map $\Spec (\Z) \to \Spec (\Sph)$ to the ordinary (derived) affine line over the integers.

Similarly one denotes $\GG_{m}:= \Spec(\Sph[\Z])$; we refer to this as the flat multiplicative group scheme. The group-completion map $\mathbb{N} \to \Z$ of $E_{\infty}$-spaces induces the map of algebras $\Sph[\mathbb{N}] \to \Sph[\Z]$ and therefore a map 
$$
\GG_m \to \AAA^1 
$$
 of spectral schemes. Moreover, there is an induced action
$$
\GG_m \times \AAA^1 \to \AAA^1,
$$
which arises from the coaction 
$$
\Sph [\mathbb{N}] \to \Sph[\mathbb{N}] \otimes \Sph[\Z]
$$
of spectra; we define $\AAA^1 / \GG_m$ to be the quotient (in the $\infty$-category 
$$
\Shv_{fpqc}(\on{CAlg}^{\on{cn}}) \subset \on{Fun}(\on{CAlg}^{cn}, \mathcal{S})
$$
of fpqc sheaves on $\on{CAlg}^{\on{cn}}$, see  \cite[Appendix B]{lurie2016spectral}) with respect to this action. More precisely, one obtains a simplicial object $\mathcal{A}_\bullet$ given in degree $n$ by 
$$
\mathcal{A}_n = \AAA^1 \times \GG_m \times...\times \GG_m
$$
  with the face and degeneracy maps being the expected ones. Then
$\AAA^1 / \GG_m$ is simply defined to be the realization $||\mathcal{A}_\bullet||$ of this simplicial object. 

\begin{rem} \label{remarkgeometric}
The stack $\AAA^1/ \GG_m$ is not a spectral Deligne-Mumford stack; it is however,  \emph{geometric}. For the purposes of this note, we use the following characterization of geometric stacks from  \cite[Corollary 9.3.1.4]{lurie2016spectral}: a geometric stack $\EuScript{X}$ arises as the geometric realization $\EuScript{X}= |\EuScript{X}^{\bullet}|$ of a  simplicial object $\EuScript{X}_{\bullet}$ in the $\infty$-category $\Shv_{fpqc}(\on{CAlg}^{\on{cn}}) \subset  \on{Fun}(\on{CAlg}^{\on{cn}},\mathcal{S})$ of sheaves in the fpqc topology on $\on{CAlg}^{\on{cn}}$, where $\EuScript{X}_0$ is affine and $d_0 : \EuScript{X}_1 \to \EuScript{X}_0$ is a representable, affine, faithfully flat map. In this case, the map  $d_0: \AAA^1 \times \GG_m \to \AAA^1$ is  faithfully flat as it is the base change of the faithfully flat morphism of affine schemes  corresponding to $\Sph \to \Sph[\Z]$; here faithful flatness corresponds to the  classical condition that $\pi_0(-)$ applied to this map is faithfully  flat. The map in question will be  affine as any morphism of affine schemes is affine (cf. \cite[Proposition 2.4.4.3]{lurie2016spectral}). 
\end{rem}

The $\infty$-category of quasi-coherent sheaves $\QCoh(\EuScript{X})$ is particularly well behaved for a geometric stack $\EuScript{X}$ (such as the ones we will be dealing with). The following proposition collects several of the properties of $\QCoh(\EuScript{X})$ which we'll be making use of:

\begin{prop} (cf. \cite[Corollary 9.1.3.2]{lurie2016spectral}) \label{quasicoherentsheavesonstack}
Let $\EuScript{X}: \on{CAlg}^{\on{cn}} \to \mathcal{S}$ be a geometric stack. Then
\begin{enumerate}
    \item The $\infty$-category $\QCoh(\EuScript{X})$ is presentable
    \item There exists a symmetric monoidal product on $\QCoh(\EuScript{X})$ (this does not require geometricity)
    \item There exists a  left and right complete $t$-structure $(\QCoh(\EuScript{X})_{\geq 0},\QCoh(\EuScript{X})_{\leq 0})$ on $\QCoh(\EuScript{X})$, which is moreover compatible with the symmetric monoidal structure on $\QCoh(\EuScript{X})$.
\end{enumerate}
\end{prop}

\begin{rem} \label{atlasstuff}
These properties essentially follow from the fact that there exists a faithfully flat morphism $f: X_0 \to \EuScript{X}$  where $X_0= \Spec(A)$ for $A$ an $E_{\infty}$ ring. Let $X_{\bullet}$ denote the Cech nerve of $f$. Then $\QCoh(\EuScript{X})$ may be defined as the totalization of the cosimplicial diagram of presentable, stable $\infty$-categories $\QCoh(X_{\bullet})$. We remark further that by \cite[Remark 9.1.3.4]{lurie2016spectral}, the pullback $f^*: \QCoh(\EuScript{X}) \to \Mod_A$ will be t-exact and conservative. Hence, $\mathcal{F} \in \QCoh(\EuScript{X})_{\geq 0}$ if and only if $f^* (\mathcal{F}) \in (\Mod_A)_{\geq 0} $
\end{rem}

\section{Filtered and Graded Spectra}
\label{azumaya}

In this section we give a quick overview  of graded and filtered spectra. Much of this can be found in \cite{lurie2014rotation}. 

\begin{defn}
Let $\EuScript{C}$ be a $\infty$-category; let $\Z^{ds}$ denote the integers viewed as a discrete space;  we let $\on{Fun}(\Z^{ds}, \EuScript{C})$ be the $\infty$-category of graded objects of $\EuScript{C}$;
in particular we set
$$
\on{Rep}(\Z^{ds})= \on{Fun}(\Z^{ds}, \Sp)
$$
to be the stable $\infty$-category of graded spectra. 
\end{defn}

\begin{defn}
Let $\Z$ be the integers viewed as an $\infty$-category via its partial ordering, and let $\EuScript{C}$ denote an arbitrary $\infty$-category. Then the $\infty$-category of filtered objects of $\EuScript{C}$ is denoted by $\on{Fun}(\Z^{op}, \EuScript{C})$;
in particular we set 
$$
\on{Rep}(\Z):= \on{Fun}(\Z^{op}, \Sp)
$$
to be the $\infty$-category of filtered spectra. 
\end{defn}

An object in $\Rep(\Z) $ is given by a sequence of spectra $X_n$ with maps
$$
... \rightarrow X_n \rightarrow X_{n-1} \rightarrow .... 
$$
As there is an inclusion of simplicial sets $\Z^{ds} \into \Z$, there exists an induced resriction functor $\on{Res}: \Rep(\Z) \to \Rep(\Z^{ds})$, which forgets the structure maps of the filtration. This has a left adjoint, $I: \Rep(\Z^{ds}) \to  \Rep(\Z)$ which sends a functor $\Z^{ds} \to \Sp$ to its Kan extension along $\Z^{ds} \into \Z$. This functor is concretely given by 
$$
(IX)_n= \bigoplus_{ m \geq n} X_m
$$

\begin{rem}
Via the symmetric monoidal structures of $\Sp$, $\Z^{ds}$ and $\Z$ --the latter two arising from the standard abelian group structure on the integers-- one endows $\Rep(\Z^{ds})$ and $\Rep(\Z)$ with a symmetric monoidal structure arising from Day convolution. For more on the Day convolution product in the setting of $\infty$-categories, see \cite{glasman2016day}.  
\end{rem}

\begin{rem}
 $\Rep(\Z)$ is monadic over $\Rep(\Z^{ds})$, cf. \cite[Proposition 3.1.6]{lurie2014rotation}. Let $\Sph^f$ denote the unit filtered spectrum, its underlying graded spectrum $\Sph[t]:= \on{Res}(\Sph^f)$ is given by 

$$
\Sph[t]_n = 
\begin{cases}
\Sph & i \leq 0 \\
      0 & i > 0 
\end{cases}
$$
This will be an $E_{\infty}$-algebra as it is the image of the unit of $\Rep(\Z)$ under the lax symmetric monoidal functor $\on{Res}$. By  loc cit. there is a symmetric monoidal equivalence 
$$
\Rep(\Z) \simeq \Mod_{\Sph^f}(\Rep(\Z)) \simeq \Mod_{\Sph[t]}(\Rep(\Z^{ds})),
$$
and so $\Rep(\Z)$ may be identified with algebras over the monad $T(-) \simeq -\otimes_{} \Sph[t]$ on $\Rep(\Z^{ds})$. 

\end{rem}

\begin{const}
Fix $i$; for $(X_{n})$ a filtered object, we let $\on{gr}(X)_i = \on{cofib}(X_{i+1} \to X_{i})$. These assemble for all $i \in \Z$, to form a functor
$$
\on{gr}: \Rep(\Z) \to \Rep(\Z^{ds});
$$
this is commonly referred to as the \emph{associated graded} of the filtered spectrum $(X_n)$.
\end{const}

\begin{rem} \label{sven}
Let $\AAA$ denote the filtered spectrum which is $\Sph$ concentrated in weight $0$, and $0$ everywhere else. This is an $E_{\infty}$-algebra object in $\Rep(\Z)$. As it turns out, $\Rep(\Z^{ds})$ can itself be described as algebras over a certain monad on $\Rep(\Z)$, by way of the associated graded functor. Indeed, by \cite[Proposition 3.2.7]{lurie2014rotation} this induces a symmetric monoidal equivalence
$$
\Rep(\Z^{ds}) \simeq \Mod_{\AAA}(\Rep(\Z)).
$$
\end{rem}

\section{Quasi-coherent sheaves on $B \GG_m$}
\label{suhjew}
Here we begin a detailed study of the category of quasi-coherent sheaves on the spectral geometric stack $B \GG_m$. Again, $\GG_m := \Spec(\Sph[\Z])$ where $\Sph[\Z] := \Sigma_{+}^{\infty}\Z^{ds}.$ The bialgebra structure on $\Sph[\Z]$  allows for one to define,
in the $\infty$-topos $\Shv_{\on{fpqc}}(\on{CAlg}^{cn})$ of fpqc sheaves on the site corresponding to connective $E_\infty$-rings, the  group object $\mathcal{G}_{\bullet}$ (cf. \cite[Section 6.1.2]{lurie2009}
$$
\mathcal{G}_n= \GG_m^{\times n} 
$$ 
As a group object is in particular a simplicial object, we may define 
$$
B \GG_m := || \mathcal{G}_{\bullet}||
$$
as the realization of this simplicial object, i.e. as the \emph{classifying stack} of $\GG_m$. 

Our goal is to prove the following identification with graded spectra:
\begin{thm} \label{benchen}
There is an equivalence of stable symmetric monoidal $\infty$-categories:
$$\QCoh(B\GG_m) \simeq \on{Rep}(\Z^{ds})$$
\end{thm}
A similar result, over $\Spec(\Z)$, appears in \cite[Section 10]{spitzweck2010derived}. 
We first describe the functor which realizes the above equivalence. Let $\pi: \Z^{ds} \to pt$ be the terminal morphism in $\EuScript{S}$, the $\infty$-category of spaces. As described in \cite{umkehr} there is an induced functor $\pi^{*}: \Sp \to \Pre_{Sp}(\Z)= \Rep(\Z^{ds})$. This functor sends a spectrum $E$ to the ``constant presheaf", namely to the functor $F: \Z^{ds} \to \Sp $ sending all $n\in \Z$ to $E$.  

Furthermore $\pi^*$ has the usual global sections right adjoint $\pi_{*}$, which sends a given $X(n) \in \Rep(\Z^{ds})$ to 
$$
\pi_*(X(n))= \prod_{n\in \Z}X(n).
$$
In addition it comes equipped with a left adjoint functor $\pi_{!}: \Pre_{Sp}(\Z^{ds}) \to Sp$ sending  $X(n) \in \Rep(\Z^{ds})$ to 
$$
\pi_!(X(n))= \bigoplus_{n\in \Z}X(n).
$$
This is op-lax monoidal with respect to the pointwise symmetric monoidal structure on $\Pre_{Sp}(\Z^{ds})$ but it is symmetric monoidal with respect to the Day convolution product. Indeed, by e.g.  \cite[Proposition 6.12]{umkehr}, Day convolution endows the  functor
$$
\Pre_{Sp}(-): \mathcal{S} \to  \EuScript{P}r_{L, st},
$$
given by 
$$
X \mapsto \Pre_{Sp}(X), \, \, \, \,  \, \,  f \mapsto f_!
$$ 
with a symmetric monoidal structure. Thus, in this particular case, $f = \pi: \Z \to pt.$ is sent to map of  $E_\infty$ algebra objects in $ \EuScript{P}r_{L, st}$, i.e a symmetric monoidal functor.

\begin{prop} \label{saywatup}
The  functor $\pi_{!}$ is comonadic; in particular there is an induced equivalence 
$$
\on{Rep}(\Z^{ds}) \simeq \coMod_{\Sph[\Z]}
$$
\end{prop}

\begin{rem}
The argument furnished here applies more generally; for $M$ an arbitrary abelian group, there will be an equivalence 
$$
\coMod_{\Sph[M]} \simeq \Rep(M^{ds}),
$$
between comodules over $\Sph[M]$, and local systems of spectra on the abelian group $M$ viewed as a discrete space.  In fact this gives a nice conceptual interpretation of the symmetric monoidal structure on comodules over the Hopf algebra $\Sph[M]$; it is the one induced via this equivalence, by the Day convolution symmetric monoidal structure on $\Rep(M^{ds})$. 
\end{rem}

The proof of this proposition boils down to an application of the following comonadic form of the Barr-Beck theorem which we reproduce here for the reader's convenience. 

\begin{thm}[Barr-Beck-Lurie] \label{barrbeck}
The adjunction $F: C \rightleftarrows D: G$ is comonadic if and only if the following hold
\begin{enumerate}
    \item The functor $F$ is conservative.
    \item Given an $F$-split cosimplicial object $X^{\bullet}$ in $C$, the diagram $X^{\bullet}$  admits a limit which will be preserved by $F$.
\end{enumerate}
\end{thm}

\begin{proof}[Proof of Proposition \ref{saywatup}]
Since $\pi_{!}$ is op-lax monoidal with respect to the pointwise symmetric monoidal structure on $\Pre_{Sp}(\Z^{ds})$, it sends comodules over the constant spectrum (the unit w.r.t this symmetric monoidal structure) to comodules over $\pi_{!}(\mathbf{1}) \simeq \Sph[\Z]$. Hence, $\pi_{!}: \Pre_{Sp}(\Z^{ds}) \to Sp$ factors through the forgetful functor $F: \on{coMod}_{\Sph[\Z]} \to Sp$.
We first show that $\pi_{!}$ is conservative. As we are dealing with a functor of stable $\infty$-categories, it is enough to show that if $\pi_{!}(X)= \oplus_{n \in \Z} X(n) \simeq 0$, then $X \simeq 0$. Note that every $X(n)$ will be a retract of the direct sum; indeed the following composition gives the identity:
$$
X(i) \to \bigoplus_{n \in \Z} X(n) \to \prod_{n \in \Z} X(n) \to X(i).
$$
Now if  $\oplus_{n \in \Z} X(n) \simeq 0$, this would then imply that every $X(i) \simeq 0$, which allows us to conclude that the graded spectrum $X \simeq 0$.

Next, we show that every $\pi_{!}$-split cosimplicial object in $\Rep(\Z^{ds})$ splits and, so, becomes a limit diagram. Let $X^{\bullet}$ be such an object. As an object in $\EuScript{P}r_{L,st}$, the $\infty$-category of presentable stable $\infty$-categories, there is an equivalence  
\begin{equation}\label{repz}
\Rep(\Z^{ds}) \simeq \prod_{i \in \Z} \Sp; 
\end{equation}
hence specifying the cosimplicial object $X^{\bullet}$ is equivalent to specifying a cosimplicial object $X^{\bullet}(i): \Delta \to Sp$ for each integer value $i$. It will be enough to show that each $X^{\bullet}(i)$ splits. We remark further that each $X^{\bullet}(i)$ will be a retract of $X^{\bullet}$; composing with $\pi_{!}$, one obtains each $\pi_{!}X^{\bullet}(i)$ as a retract of $\pi_{!} X^{\bullet}$ in the $\infty$-category of split cosimplicial diagrams $\on{Fun}(\Delta_{- \infty}, \Sp)$. Note that each $\pi_{!}X^{\bullet}(i)$ is canonically (co)augmented by the limit. We now appeal to  \cite[Corollary 4.7.2.13]{lurie2016} which states that for an $\infty$-category $\EuScript{C}$, the subcategory $\EuScript{X}$ of augmented cosimplicial objects in $\on{Fun}(\Delta, \EuScript{C})$ that extend to split augmented cosimplicial objects is closed under taking retracts in the $\infty$-category $\on{Fun}(\Delta, \EuScript{C})$; moreover they remain retracts in the $\infty$-category of split augmented cosimplicial objects.  Hence each $\pi_{!}X^{\bullet}(i)$ splits as a cosimplicial object in $\Sp$ and is a limit diagram. Of course, as this is true for all $i$, we can take the product  to obtain a split augmented cosimplicial object $\tilde{X}^{\bullet} \in \on{Fun}(\Delta_{- \infty}, \Rep(\Z^{ds}))$ extending $X^{\bullet}$. 
It follows that $\pi_{!}$ is comonadic, and so we may conclude that 
\[
\Rep(\Z^{ds}) \simeq \coMod_{\pi_{!} \pi^{*}} \Sp \simeq  \on{coMod}_{\Sph[\Z]}.
\]
\end{proof}

\begin{proof} [Proof of Theorem \ref{benchen}]
We now prove the main theorem of this section. Recall that $B\GG_m$ is \emph{geometric}; as such it has an atlas, in this case given by the faithfully flat map $e: Spec(\Sph) \to B\GG_m$, and we can define $\QCoh(B \GG_m)$ as the totalization of the following cosimplicial diagram of $\infty$-categories:
$$
\Sp \rightrightarrows \Mod_{\Sph[\Z]} \substack{\rightarrow\\[-1em] \rightarrow \\[-1em] \rightarrow} \Mod_{\Sph[\Z] \otimes \Sph[\Z]} ...
$$

\noindent Note that this cosimplicial diagram may be coherently augmented  via the functor $\pi_{!}: \Rep(\Z^{ds}) \to \Sp$. More precisely, we obtain, in each cosimplicial degree $n$, the assignment 
$$
X \mapsto (\pi_! \pi^*)^{\circ n} \pi_!(X) \simeq  \Sph[\Z]^{n} \otimes_{\Sph} \pi_! (X)
$$
One sees by the standard tensor product associator formulas  that this agrees with $\pi_!$ composed with the cosimplicial structure maps. Thus we have produced an augmented cosimplicial  $\infty$-category.
By our above arguments this satisfies the conditions of the following comonadic version of  (\cite[Corollary 4.7.5.3]{lurie2016}):

\begin{prop} \label{barbecklimit}
Let $\EuScript{C}^{\bullet}: N(\Delta_+) \to Cat_{\infty}$ be an augmented cosimplicial  $\infty$-category and set $\EuScript{C}^{-1} = \EuScript{C}.$ Let $G: \EuScript{C} \to  \EuScript{C}^0$ be the canonical augmentation functor. Assume that
\begin{enumerate}
    \item $G$ is conservative
    \item The $\infty$-category  $\EuScript{C}^{-1}$ admits totalizations of $G$-split cosimplicial objects, and those totalizations are preserved by $G$. 
    \item For every morphism $\alpha: [m] \to [n]$ in $\Delta_{+}$, the diagram 
    $$
\xymatrix{
& \EuScript{C}^m \ar[d] \ar[r]^{d^0} &  \EuScript{C}^{m+1} \ar[d]\\
 & \EuScript{C}^n  \ar[r]^{d^0}& \EuScript{C}^{n+1}
}
$$
is right adjointable.
Then, the canonical map $\theta: \EuScript{C} \to \lim_{n\ \in \Delta} \EuScript{C}^n$ is an equivalence. 
\end{enumerate}
\end{prop}

Indeed,  $G= \pi_!$ is comonadic and moreover, all the coface maps are right adjointable (this is a consequence of the geometricity of $B \GG_m$), so that by the proposition,  there is an equivalence 
\begin{equation}\label{relevantequivalence}
\Rep(\Z) \simeq \lim \Mod_{\Sph[\Z_{+}]^{\bullet} } \simeq   \QCoh(B \GG_m).   
\end{equation}
Furthermore, we claim that this is an equivalence of symmetric monoidal $\infty$-categories, that is, it may be promoted to an equivalence in the $\infty$-category  $\on{CAlg}(\EuScript{P}r^{L, st})$.  To see this, note that each of the structure maps of the cosimplicial diagram is given by pullback functors $f^*: \QCoh(X) \to \QCoh(Y)$ at the level of quasi-coherent sheaves which are symmetric monoidal. This makes $\EuScript{C}^\bullet= \Mod_{\Sph[\Z]}$ into a limit diagram in $\on{CAlg}(\EuScript{P}r^{L, st})$  We may now apply the fact that limits in $\on{CAlg}(\EuScript{P}r^{L, st})$  commute with the forgetful functor and so may be computed in the underying category $\EuScript{P}r^{L, st}$. Putting all this together, we conclude that  the equivalence (\ref{relevantequivalence}) is an equivalence in $\on{CAlg}(\EuScript{P}r^{L, st})$.
\end{proof}

\begin{rem}
We remark that Proposition \ref{barbecklimit} is stated in a slightly different form in \cite{lurie2016}, and thus the proof of the corresponding result must be modified to obtain Proposition \ref{barbecklimit}. Indeed, the proposition in loc cit. utilizes the \emph{monadic} form of Barr-Beck to conclude that the map  $G: \EuScript{C}^{-1} \to \EuScript{C}^0$ admits a \emph{left} adjoint $F$ such  $\emph{C}^{-1}$ is equivalent to left modules $\on{Mod}_T(\EuScript{C})$, over the monad $T= G \circ F$.  Here instead, $F$ will be obtained as a \emph{right} adjoint of $G$ and the assumption that totalizations of $G$-split cosimplicial objects are preserved by $G$ allows for one to apply the comonadic form of the Barr-Beck theorem (Theorem \ref{barrbeck}) to identify $\lim_{n \in \Delta} \EuScript{C}^n$ with \emph{comodules} $\on{coMod}_T(\EuScript{C}^0)$ of the corresponding comonad $T = G \circ F$  on $\EuScript{C}^0$. 
\end{rem}

\section{Proof of main theorem}
\label{proofsection}
We now conclude the proof that $\QCoh( \mathbb{A}^1 / \GG_{m}) \simeq \Rep(\Z)$. We first describe the strategy. Recall from \cite[Proposition 3.1.6]{lurie2014rotation} that the $\infty$-category of filtered spectra is monadic over graded spectra. Indeed, there exists an $E_{\infty}$-algebra $\Sph[t] \in \Rep(\Z^{ds})$ for which 
$$
\Mod_{\Sph[t]}(\Rep(\Z^{ds})) \simeq \Rep(\Z).
$$
Hence, the functor $\Rep(\Z) \to \Rep(\Z^{ds})$ 
Similarly, $\QCoh( \AAA^1 / \GG_m)$ is monadic over  $\QCoh(B \GG_m)$. The pushforward functor $\phi_*: \QCoh (\AAA^1 / \GG_m) \to \QCoh(B \GG_m)$ induces a symmetric monoidal equivalence
$$
\QCoh(\AAA^1 / \GG_m) \simeq \Mod_{\phi_* \OO_{\AAA^1 / \GG_m}}(\QCoh(B \GG_m))
$$
so that $ \QCoh (\AAA^1 / \GG_m)$ is identified with algebras over the monad $- \otimes \phi_*(\OO_{\AAA^1 / \GG_m})$ on $\QCoh(B \GG_m)$.

Note that this is not true for an arbitrary map of stacks; in this case it follows (cf. \cite[Proposition 3.2.5]{lurietannakadag}) from the fact that $\phi: \AAA^1 / \GG_m \to B \GG_m$ is a quasi-affine representable map. Recall that morphism $f: \EuScript{X} \to \EuScript{Y}$ spectral stacks is quasi-affine if for any quasi-affine spectral scheme $Y \to \EuScript{Y}$, the fiber product $Y \times_{\EuScript{Y}} \EuScript{X}$ is quasi-affine. By Remark 3.1.29 of loc. cit., this property holds for $f : \EuScript{X} \to \EuScript{Y} $ with $\EuScript{Y}\simeq \colim \EuScript{Y}_i$ if it holds for each pullback $f_i: \EuScript{X} \times_{\EuScript{Y}} \EuScript{Y}_i$. In our case with $\EuScript{Y}= B \GG_m$, the pullback  of the map $\phi$ in each degree is the projection map 
$$
\AAA^1 \times \GG_m^{\times n} \to \GG_m^{\times n} 
$$
which is an affine map.

In order to furnish the desired equivalence, it will be enough to identify $\phi_{*} \OO_{\AAA^1 / \GG_m}$ with \\ $\on{Res}(\Sph^f) \simeq \Sph[t]$ as graded $E_{\infty}$-algebras, equivalently by Proposition \ref{saywatup}, as comodules over $\Sph[\Z]$. 

We recall from Remark \ref{remarkgeometric} that the quotient map $p' : \mathbb{A}^1 \to  \AAA^1/ \GG_{m}$ is an atlas for $\AAA^1/ \GG_{m}$, making it into a geometric stack. In particular, it is a flat and surjective map of spectral stacks. Surjectivity is clear, and flatness can be checked by pulling back (as we did for the quasi-affine map) along each map $\mathcal{A}_i \to \AAA^1 / \GG_m$ in the simplicial diagram. Moreover, we have the following cartesian square of stacks  where the horizontal maps are atlases 
$$
\xymatrix{
&\AAA^1 \ar[d]_{\phi'} \ar[r]^{p'} &  \AAA^1/ \GG_m\ar[d]^{\phi}\\
 & Spec(\Sph)   \ar[r]^{p}& B \GG_m
}
$$
The induced
diagram on $\infty$-categories of quasi-coherent sheaves

$$
\xymatrix{
&\QCoh(B \GG_m) \ar[d] \ar[r] &  \Sp \ar[d]\\
 & \QCoh(\AAA^1/\GG_m)  \ar[r]& \QCoh(\AAA^1)
}
$$
is right-adjointable, by \cite[Proposition 6.3.4.1]{lurie2016spectral} so that the following Beck-Chevalley transformation
$$
p^* \phi_* \to \phi'_* p'^*
$$
is an equivalence of functors $\QCoh(\AAA^1/\GG_m) \to \Sp$. Of course
$$
\phi'_* p'^* (\OO_{\AAA^1/\GG_m}) \simeq \phi'_{*}\OO_{\AAA^1} \simeq \Sph[\mathbb{N}];
$$
where  $\Sph[\mathbb{N}] \simeq \Sigma^{\infty}_+ \mathbb{N}$, the $E_{\infty}$-algebra satisfying $\pi_{*} \Sph[\mathbb{N}] \simeq (\pi_{*}\Sph)[t]$ on homotopy.

It is clear therefore that the underlying spectrum of $\phi_{*}   \OO_{\AAA^1/ \GG_m}$ agrees with the underlying spectrum of the graded spectrum $\Sph[t]$. Hence it will be enough to  show that the two  induced $\Sph[\Z]$-comodule structures on the spectrum $\Sph[\mathbb{N}]$  are equivalent. 
We first analyze the comodule structure on $\pi_!(\Sph[t])$. Let $\iota: \mathbb{N} \to \Z$ be the standard monoidal inclusion of the natural numbers into the integers. Via our previous identification of graded spectra with $\Sph[\Z]$-comodules, we see that that the canonical $\Sph[\Z]$-comodule structure on $\pi_!(\Sph[t])$ is induced by this map of spaces. Indeed, $\Sph[t]$ is the graded spectrum associated to the paramatrized space $\iota \in \mathcal{S}_{/\Z}$  via the infinite suspension functor 
$$
\Sigma^{\infty}: \mathcal{S}_{/ \Z } \simeq \on{Fun}(\Z^{op}, \mathcal{S}) \to \on{Fun}(\Z^{op}, \Sp) = \Rep(\Z^{ds}) 
$$
\noindent The comodule structure map 
$$
\Sph[\mathbb{N}] \to \Sph[\mathbb{N}] \otimes \Sph[\Z]
$$
which arises canonically via Proposition \ref{saywatup} on     
$\Sph[\mathbb{N}_+] \simeq \pi_! \circ\Sigma^{\infty}(\iota)$
is the suspension of the map $\mathbb{N} \xrightarrow{(id, \iota)} \mathbb{N} \times \Z$. 

Next, we investigate the natural $\Sph[\Z_+]$-comodule structure on  
the spectrum $\Sigma^{\infty}_+ \mathbb{N}$ which arises via the identification
$$
f^*\phi_{*} \OO_{\AAA^1 / \GG_m} \simeq \Sigma^{\infty}_+ \mathbb{N},
$$
(here $f^{*}$ is the pullback map along $\Spec (\Sph) \to B \GG_m$) This arises canonically from the $\GG_m$-equivariant structure on $\OO_{\AAA^1/ \GG_m}$ via the action map 
$$
\AAA^1 \times \GG_m \to \AAA^1.
$$
The dual co-action map
$$
\OO_{\AAA^1}\simeq \Sigma^{\infty} \mathbb{N} \to \Sph[\Z] \otimes \Sph[\mathbb{N}] 
$$
is, by definition, the suspension of the map  of $E_{\infty}$-spaces $\mathbb{N} \to \mathbb{N} \times \mathbb{Z}$ given by
$$
n \mapsto (n,n)
$$
This is precisely the comodule structure map on $\pi_! \Sph[t]$ described above.  
Hence, the above equivalence $f^{*} \phi_{*} \OO_{\AAA^1/ \GG_m}\simeq \Sigma^{\infty} \mathbb{N}_+$ may be promoted to an equivalence of $\Sph[\Z]$-comodules. 

We have proven the following:

\begin{prop} \label{gradedalgebra}
The equivalence $\pi_{!}: \Rep(\Z^{ds}) \to  \QCoh(B\GG_m)$ sends the graded spectrum $\Sph[t]$ to  $\phi_{*}(\OO_{\AAA^1/\GG_m})$.
\end{prop}

\begin{proof}[Proof of Theorem \ref{main}]
By Theorem \ref{benchen},there is an equivalence 
$$
\Phi: \Rep(\Z^{ds}) \simeq \QCoh(B \GG_m). 
$$
Since this is symmetric monoidal, it will send $\Sph[t]$-modules to  $\phi_{*}(\OO_{\AAA^1/\GG_m})$-modules in $\QCoh(B \GG_m)$, and this identification will be symmetric monoidal as well. Indeed, one obtains the following  commutative diagram  (cf. \cite[Theorem 4.5.3.1]{lurie2016}
$$
\xymatrix{
&\Mod(\Rep(\Z^{ds}))^{\otimes} \ar[d]^{} \ar[r]^{} &  \Mod(\QCoh(B \GG_m))^{\otimes} \ar[d]^{}\\
 & \on{CAlg}(\Rep(\Z^{ds}))\times N(\Gamma_*)   \ar[r]^{} & \on{CAlg}(\QCoh(B \GG_m))\times N(\Gamma_*)
}
$$
where the top horizontal maps are equivalences of (generalized) $\infty$-operads and the vertical maps are coCartesian fibrations. On each sides, the fiber over some fixed $E_\infty$-algebra object $A$ is precisely the symmetric monoidal $\infty$-category $\Mod_A$. 
By the above analysis, the equivalence $\Phi$ sends $\Sph[t]$ to the push-forward quasi-coherent sheaf $\phi_{*}(\OO_{\AAA^1/\GG_m})$. Thus, the fiber over $\Sph[t] \times N(\Gamma_*)$ will be sent to the fiber  over  $\phi_* \OO_{\AAA^1 / \GG_m} \times N(\Gamma_*)$. 
Putting all this together, we deduce an equivalence  
$$
 \Mod_{\Sph[t]}(\Rep(\Z^{ds}))^{\otimes} \simeq  \Mod_{\phi_{*}(\OO_{\AAA^1/\GG_m})}(\QCoh( B \GG_m))^{\otimes}
$$
of symmetric monoidal $\infty$-categories.

\end{proof}

\begin{rem}
The above argument is valid if we replace the sphere with any $E_{\infty}$ ring spectrum $R$. In particular, if $R = \Z$ we recover the integral statement appearing without proof in \cite{ben2012loop}. 
\end{rem}

\section{Generic and closed points}
In this section, we show that we can interpret the underlying spectrum and associated graded functors as pullback via certain morphisms of spectral stacks.  We shall use the framework of Tannaka duality over a \emph{locally Noetherian} geometric stack, described by Lurie in \cite[Section 9.5]{lurie2016spectral}. This means that there exists a faithfully flat (i.e flat and surjective) map $\Spec (R) \to X$, where $R$ is a Noetherian $E_\infty$-ring. For the reader's convenience, we recall the defintion of Noetherian in this context:

\begin{defn}
Let $R$ be an $E_\infty$ ring. We say $R$ is coherent if $\pi_0 R$ if every finitely generated ideal of $\pi_0(R)$ is finitely presented and each homotopy group $\pi_n R$ is a finitely presented $\pi_0(R)$ module. We say $R$ is Noetherian if $R$ is coherent and $\pi_0(R)$ is Noetherian.  
\end{defn}

One sees (essentially by an application of the Hilbert basis theorem)  that  $\Sph[\mathbb{N}]$ is a Noetherian $E_\infty$-ring. As the map $\pi: \AAA^1 \to \GG_m$ is faithfully flat (recall that this can be tested by pulling back the map to each level of the simplicial diagram in which case the map is clearly faithfully flat) this makes   $\AAA^1 /\GG_m$ is a locally Noetherian geometric stack.   

We freely use the following result: 

\begin{thm}[Lurie]
Let $X, Y: \on{CAlg}^{\on{cn}} \to \mathcal{S}$ be functors and suppose $X$ is a  locally Noetherian geometric stack. Then the construction $f: Y \to X \mapsto f^*: \QCoh(X) \to \QCoh(Y)$ determines a fully faithful embedding 
$$
\Map_{}(Y, X) \to \on{Fun}_{\on{Fun}^{\otimes}(\on{CAlg}^{\on{cn}}, \mathcal{S} )}(\QCoh(X), \QCoh(Y))
$$
whose essential image is spanned by symmetric monoidal functors $F: \QCoh(X) \to \QCoh(Y)$ which preserve small colimits and connective objects. 
\end{thm}
\noindent First we analyze what occurs with the ``underlying spectrum" functor: 
\begin{prop} \label{generic point}
There exists a unique map of stacks $1: \Spec (\Sph) \to \AAA^1 / \GG_m$ which, upon taking pullback, recovers the functor 
$$
U : \Rep(\Z) \to \Sp,  \, \, \, \, \, U((\EuScript{F}_i)_{i \in \Z}) = \colim_{i \in \Z} \EuScript{F}_i
$$
which associates to a  filtered spectrum, its underlying object. We refer to this as the generic point of $\AAA^1 / \GG_m$.
\end{prop}

\begin{rem}
One may think of this map as the $\GG_m$-equivariant inclusion $\GG_m \to \AAA^1$
where $\GG_m$ acts transitively on itself. More precisely, this may be viewed as a morphism of groupoid objects $\widetilde{\mathcal{G}_\bullet} \to \mathcal{A}_{\bullet}$ where in each simplicial degree $n$ is the map 
$$
\GG_m \times \GG_m^{n} \to \AAA^1 \times \GG_m^{\times n}
$$
given by the standard inclusion of $\GG_m$ into $\AAA^1$ on the first factor. Upon taking geometric realization of this morphism of groupoid objects, one obtains  the relevant map
$$
\Spec (\Sph) \simeq \GG_m / \GG_m \to \AAA^1 / \GG_m 
$$
\end{rem}

\begin{proof}
The proof is an application of a form of Tannaka duality as it appears in \cite{lurie2016spectral}. Let $U: \Rep(\Z) \to \Sp$, be the underlying spectrum functor. This is symmetric monoidal, which can be seen from the fact that the smash product in spectra commutes with the relevant colimits.  As a functor out of  the equivalent $\infty$-category $\QCoh(\AAA^1 / \GG_m) $ it is symmetric monoidal, and sends connective objects to connective objects in $\Sp$. Here we use the $t$-structure on $\QCoh(\AAA^1 / \GG_m)$ arising from the fact that it is a geometric stack and the connective objects $\EuScript{F} \in \QCoh(\AAA^1  / \GG_m)^{cn}$ are those which pullback to connective objects along the faithfully-flat cover $\AAA^1 \to \AAA^1 / \GG_m$.  By the considerations of the previous section, we may interpret this functor as the composition 
$$
\QCoh(\AAA^1 / \GG_m) \simeq \Mod_{\phi_{*}\OO_{\AAA^1/ \GG_m}}(\QCoh (B \GG_m)) \to \Mod_{\Sigma^{\infty}_+ \mathbb{N}}(\Sp)
$$
which sends a filtered spectrum viewed as a $\Sph[t]$-module in graded spectra to its underlying $\Sigma^{\infty}_+ \mathbb{N}$ module in spectra. Hence, a filtered spectrum $\{...\to \EuScript{F}_i  \to \EuScript{F}_{i-1} \to ... \}$ will be connective if and only if this spectrum  $\oplus_{i \in \Z} \EuScript{F}_i$ is connective, which in turn holds only if each filtered piece $\EuScript{F}_i$ is connective. Since each $\EuScript{F}_i$ is now assumed to be connective, and since connective objects are closed under colimits, we conclude that $U(\{\EuScript{F}_i\}) = \colim(...\to \EuScript{F}_i  \to \EuScript{F}_{i-1} \to ... )$ is connective. The criteria of \cite[Proposition 9.5.4.1]{lurie2016spectral}  are verified so that this functor corresponds to a unique point $1: \Spec (\Sph) \to \AAA^1 / \GG_m $.   
\end{proof}
\noindent Next we analyze the associated graded functor: 
\begin{prop} \label{closed point}
There exists a unique map of stacks $0: B \GG_m \to \AAA^1 / \GG_m $ which induces, via pullback, the functor 
$$
\gr: \Rep(\Z) \to \Rep(\Z^{ds})
$$
We refer to this as the closed point.
\end{prop}

\begin{rem}
One may think of this map as the $\GG_m$-equivariant inclusion 
$$
\Spec (\Sph) \to \AAA^1 
$$
More precisely, this is the map of groupoid objects $\mathcal{G}_\bullet \to \mathcal{A}_{\bullet}$ 
 given in degree $n$ by 
$$
\Spec (\Sph)  \times \GG_m^{\times n} \to \AAA^1 \times \GG_m^{n}
$$
where the map between the first factors is the map of affine schemes corresponding to the morphism
$$
\Sph[\mathbb{N}] \to \Sph 
$$
of $E_\infty$ rings (induced by the final map of spaces $\mathbb{N} \to pt.)$
\end{rem}

\begin{proof}
This will be another application of Tannaka duality. Recall that $\gr: \Rep(\Z) \to \Rep(\Z^{ds})$ factors as
$$
\Rep(\Z) \xrightarrow{\otimes \AAA}\Mod_{\AAA}(\Rep(\Z)) \xrightarrow{\simeq} \Rep(\Z^{ds}),
$$
where $\AAA$ is the $E_\infty$-algebra in $\on{Rep}(\Z)$  described in Remark \ref{sven}. As a filtered object, this is the sphere spectrum $\Sph$ concentrated in weight zero, with all structure maps being trivial.  
Hence we may identify the associated graded functor,  $\gr(-)$,  with tensoring by $\AAA \in \Rep(\Z)\simeq \QCoh(\AAA^1 / \GG_m)$. We abuse notation slightly and refer to the corresponding object of $\QCoh(\AAA^1 / \GG_m)$ as $\AAA$. We claim that $\AAA$ is a connective object with respect to the t-structure on $\QCoh(\AAA^1 / \GG_m)$.   Upon pulling back to the cosimplicial diagram of categories corresponding to the faithfully flat cover $\pi: \AAA^1 \to \AAA^1 / \GG_m$,  one obtains the equivalence 
$$
\pi^*(\AAA) \simeq \Sph
$$
as objects in $\QCoh(\AAA^1) \simeq \Mod_{\Sph[\mathbb{N}]}$. Since $\Sph$ is a connective $\Sph[\mathbb{N}]$-module, we conclude that $\AAA$ is a connective object  with respect to the $t$-structure on $\QCoh(\AAA^1 / \GG_m)$. Hence, tensoring by $\AAA$ sends connective objects of $\Rep(\Z) \simeq \QCoh(\AAA^1 / \GG_m)$ to connective objects of $\Mod_{\AAA}(\Rep(\Z)) \simeq \QCoh(B \GG_m)$. As this functor is clearly symmetric monoidal and preserves small colimits, it satisfies the requirements of \cite[Proposition 9.5.4.1]{lurie2016spectral} and is therefore induced by a map of stacks $0: B \GG_m \to \AAA^1 / \GG_m$.
 
\end{proof}

\section{Filtrations on cohomology}
Let $\EuScript{X} \in \on{Pre}(\on{CAlg}^{\on{cn}})$ be a spectral stack admitting a map $f: \EuScript{X} \to \AAA^1/\GG_m$. This gives rise to a $\QCoh(\AAA^1 / \GG_m)$-action on $\QCoh(\EuScript{X})$, since the pullback functor 
$$
f^*:   \QCoh(\AAA^1 / \GG_m) \to  \QCoh(\EuScript{X})
$$
will be symmetric monoidal, making $\QCoh(\EuScript{X})$ into an $\QCoh(\AAA^1 / \GG_m)$-module object (in fact it will be an algebra object) in $\EuScript{P}r_{L,st}$ with its standard symmetric monoidal structure (cf. \cite[Section 4.8.1]{lurie2016}). Note that this pullback functor is a left adjoint so it preserves colimits; hence the action of $\QCoh(\AAA^1 / \GG_m)$ on $\QCoh(\EuScript{X})$ will itself preserve colimits in each variable. We may thus consider $\QCoh(\EuScript{X})$ as a locally filtered stable $\infty$-category, in the sense of \cite{lurie2014rotation}. We review quickly what this means. 

\begin{defn}
Let $\EuScript{C}$ be a stable $\infty$-category. Then a \emph{local filtration} (cf. \cite[Definition 3.1.10]{lurie2014rotation} on $\EuScript{C}$ corresponds to a (left) action $N(\mathbb{Z}) \times \EuScript{C} \to \EuScript{C}$. By Remark 3.1.12 of loc. cit., this is equivalent to an action of $\Rep^{\on{fin}}(\Z)$ which is separately exact in each variable. If $\EuScript{C}$ is presentable, this  induces an action of $\on{Ind}(\Rep^{\on{fin}}(\Z)) \simeq \Rep(\Z)$ which preserves all colimits separately in each variable. This makes $\EuScript{C}$ into a $\Rep(\Z)$-module in the the $\infty$-category $\EuScript{P}r_{L,st}$ of presentable stable $\infty$-categories. 

\end{defn}

\begin{rem}
As described in \cite[Remark 3.1.14]{lurie2014rotation}, one may think of a local filtration on a stable $\infty$-category $\EuScript{C}$ as additional data allowing us to view the objects as equipped with filtrations; roughly this corresponds to shift functors $X \mapsto X(n) $ together with natural maps $X(n) \to X(m)$ together with associated higher coherence data. 
\end{rem}

In favorable situations, this can be refined a bit further: a map of stacks $\EuScript{X} \to \AAA^1 / \GG_m$ gives rise  to a filtration on the $E_{\infty}$-algebra of global sections  $\R \Gamma(X ,\OO_X)$ of the underlying spectral stack $X:= \Spec (\Sph) \times_{\AAA^1 / \GG_m} \EuScript{X}$ 

\begin{prop}
Let $\EuScript{X}$ be a geometric stack admitting a map  to $f: \EuScript{X} \to \AAA^1 / \GG_m $. Assume further that one of the two conditions hold:
\begin{itemize}
    \item  $\EuScript{X}$ is representable by a spectral scheme (cf. \cite[Definition 1.1.2.8]{lurie2016spectral}); 
    \item  The morphism $f$ is cohomologically of finite dimension in the sense that for every connective $E_{\infty}$-ring $R$, and every $R$-point $\eta: \Spec (R) \to \AAA^1 / \GG_m$, the fiber product $ \EuScript{X} \times_{\AAA^1 / \GG_m} \Spec (R)$ is  of finite cohomological dimension over $\Spec (R)$; namely, the global sections functor sends $\QCoh(\EuScript{X} \times_{\AAA^1 / \GG_m} \Spec (R))_{\geq 0} \to {\Mod_R}_{\geq n}$ for some $n$.
 
\end{itemize}
Then, the global sections  $\R \Gamma(X, \OO_{X})$ of the underlying stack $X$ comes equipped with a natural filtration.   
\end{prop}

\begin{proof}
Let $f: \EuScript{X} \to \AAA^1 / \GG_m$ be a filtered stack satisfying one of the above criteria. By \cite[Proposition 9.1.5.7]{lurie2016spectral}, the diagram of stable $\infty$-categories 
$$
\xymatrix{
&\QCoh(\AAA^1 / \GG_m) \ar[d]^{1^*} \ar[r]^{f^*} &  \QCoh(\EuScript{X}) \ar[d]^{1'^*}\\
 & \Sp  \ar[r]^{f'^*} & \QCoh(X)
}
$$
is right adjointable, and so, the Beck-Chevalley natural transformation  of functors  $ 1^* f_*  \simeq   f'_*1'^{*}: \QCoh(\EuScript{X}) \to \Sp$ is an equivalence. Hence, 
$$
\R \Gamma(X, \OO_{X}) = f'_{*} \OO_{X} \simeq f'_{*} 1'^{*}\OO_{\EuScript{X}} \simeq 1^* f_* \OO_{\EuScript{X}}.
$$
Of course $f_* \OO_{\EuScript{X}}$ will, by our main theorem, now give rise to an object in $\Rep(\Z)$, hence a filtered spectrum. Moreover, we have identified the functor $1^*$ with the functor associating the underlying object of a filtration; thus  $1^* f_{*}  \OO_{\EuScript{X}}$ will be the underlying object of the filtered spectrum  corresponding to $f_* \OO_{\EuScript{X}}$. 
\end{proof} 

\section{Remarks on $t$-structures}
We end this note with  several observations about the various $t$-structures on filtered spectra. In the previous section we made use of the canonical  $t$-structure  on $\QCoh(\AAA^1 / \GG_m)$ described in Proposition \ref{quasicoherentsheavesonstack}.   This corresponds to the \emph{neutral t-structure} on $\Rep(\Z)$ which we now recall:

\begin{const}
Let $\Rep(\Z)_{\geq 0}$ be the full subcategory of $\Rep(\Z)$ consisting of filtered objects $F^*X$ for which $F^n X \in \Sp_{\geq 0}$ for all $n \in \Z$. Then $(\Rep(\Z)_{\geq 0}, \Rep(\Z)_{\leq 0})$ is a t-structure compatible with the Day convolution symmetric monoidal structure on  $\Rep(\Z)$. 
\end{const}

\begin{prop}
The neutral $t$-structure is induced by  the t-structure on $\QCoh(\AAA^1/\GG_m)$. In other words, the symmetric monoidal equivalence  of Theorem 1.1 is $t$-exact 
\end{prop}

\begin{proof} 
 We will show that (co-)connective objects of  $\Rep(\Z)$ are sent to (co)-connective objects  of $\QCoh(\AAA^1 / \GG_m)$ and vice versa. Let $\pi: \AAA^1 \to \AAA^1 / \GG_m$ denote the faithfully flat atlas for $\AAA^1/\GG_m$. The pullback functor 
$$
\pi^*: \QCoh(\AAA^1 / \GG_m) \to \QCoh(\AAA^1) \simeq \Mod_{\Sph[\mathbb{N}]}
$$
may be identified via our main theorem, with the functor assigning to a filtered spectrum  $F^*(X)$ the coproduct $\bigoplus_{\Z} F^*(X)$, which has a natural $\Sph[\mathbb{N}]$-module structure. This can be seen from the equivalence 
$$
\Rep(\Z) \simeq \Mod_{\Sph[t]}(\Rep(\Z^{ds})) \xrightarrow{\Phi} \Mod_{\Sph[t]}(\QCoh(B \GG_m)) 
$$
together with the fact that the (symmetric monoidal) pullback functor 
$$
\pi^* : \QCoh(B\GG_m) \to \Sp
$$
induced by the atlas $\Spec (\Sph) \to B \GG_m$ is itself identified with the colimit functor ($\pi_!$ in the notation of this paper). One can now easily see that the coproduct $\bigoplus_{\Z} F^*(X)$ will be connective as an $\Sph[\mathbb{N}]$-module if and only if each $F^i(X)$ is connective. The same exact arguments go through to see the coconnective part of the $t$-structure is preserved. 
\end{proof}

\begin{rem}
The $\infty$-categories $\Rep(\Z)$ and $\Rep(\Z^{ds})$ come equipped with several different $t$-structures which are distinct from the neutral $t$-structure. For example, on filtered spectra, one has the Belinson $t$-structure (as studied in for example \cite{beilinson1987derived,bhatt2019toological}) on filtered spectra. It is an interesting question whether or not one can describe this $t$-structure completely geometrically by decomposing the stack $\AAA^1 / \GG_m$ (for example by way of recollements cf. \cite{barwick2016note})). 
\end{rem}
\bibliographystyle{amsalpha}
\bibliography{biblio}
\end{document}